\documentclass[12pt]{article}
\usepackage{amsmath}
\usepackage{subfigure}
\usepackage{float}
\usepackage{amsthm}
\usepackage{amssymb, setspace}
\usepackage{color,graphicx, epsfig, geometry, hyperref,fancyhdr}
\usepackage{mathrsfs}
\usepackage{bbm}
\usepackage[T1]{fontenc}
\usepackage{latexsym,amssymb,amsmath,amsfonts}
\usepackage{color}
\usepackage{soul}
\usepackage{caption}
\usepackage{dsfont}
\usepackage{multirow}
\usepackage{rotating}
\newtheorem{theorem}{Theorem}[section]

\newtheorem{lemma}{Lemma}[section]

\newcommand{\dif}{\mathrm{d}}

\newcommand{\weakc}{\rightharpoonup}
\newcommand{\R}{\mathbb{R}}
\newcommand{\C}{\mathbb{C}}

\newcommand{\grad}{\nabla}

\newcommand{\ep}{\varepsilon}

\begin{document}
\lhead{}
\rhead{}

\begin{flushleft}
\Large 
\noindent{\bf \Large Analysis and computation of the transmission eigenvalues with a conductive boundary condition}
\end{flushleft}

\vspace{0.2in}

{\bf  \large Isaac Harris}\\
\indent {\small Department of Mathematics, Purdue University, West Lafayette, IN 47907, USA}\\
\indent {\small Email: \texttt{harri814@purdue.edu}}\\

{\bf  \large Andreas Kleefeld}\\
\indent {\small Forschungszentrum J\"{u}lich GmbH, J\"{u}lich Supercomputing Centre, Wilhelm-Johnen-} \\
\indent {\small Stra{\ss}e, 52425 J\"{u}lich, Germany. Email: \texttt{a.kleefeld@fz-juelich.de}}
\vspace{0.2in}

\begin{abstract}
\noindent We provide a new analytical and computational study of the transmission eigenvalues with a conductive boundary condition. These eigenvalues are derived from the scalar inverse scattering problem for an inhomogeneous material with a conductive boundary condition. The goal is to study how these eigenvalues depend on the material parameters in order to estimate the refractive index. The analytical questions we study are: deriving Faber-Krahn type lower bounds, the discreteness and limiting behavior of the transmission eigenvalues as the conductivity tends to infinity for a sign changing contrast. We also provide a numerical study of a new boundary integral equation for computing the eigenvalues. Lastly, using the limiting behavior we will numerically estimate the refractive index from the eigenvalues provided the conductivity is sufficiently large but unknown.
\end{abstract}

\vspace{0.1in}


\section{Introduction}
In this paper, we consider analytical and computational questions involving the interior transmission eigenvalues for a scalar scattering problem with a conductive boundary condition.  Transmission eigenvalue problems are derived by considering the inverse scattering problem of attempting to recover the shape of the scatterer from the far-field data (see for e.g. \cite{TE-book}). The scatterer is assumed to be illuminated by an incident plane wave, then the corresponding direct scattering problem associated with  the transmission eigenvalues is given by: find the total field $u \in H^1(D)$ and scattered field $u^s \in H^1_{loc}(\R^d \setminus \overline{D})$ (where $d=2,3$) such that  
\begin{align}
\Delta u^s +k^2 u^s=0  \quad \textrm{ in }  \R^d \setminus \overline{D}  \quad  \text{and} \quad  \Delta u +k^2 nu=0  \quad &\textrm{ in } \,  {D}  \label{direct1}\\
 (u^s+u^i )^+ - u^-=0  \quad  \text{and} \quad \partial_\nu (u^s+u^i )^+ + \eta (u^s+u^i )^+= {\partial_{\nu} u^-} \quad &\textrm{ on } \,  \partial D \label{direct3}\\
 \lim\limits_{r \rightarrow \infty} r^{{(d-1)}/{2}} \left( {\partial_r u^s} -\text{i} k u^s \right)=0 \label{SRC}&\,.
\end{align}
The superscripts $+$ and $-$ indicate the trace on the boundary taken from the exterior or interior of the domain $D$, respectively. The parameter $\eta$ denotes the conductivity. The total field is given by $u=u^s+u^i$ where the incident field is given by $u^i=\text{e}^{\text{i} k x \cdot \hat{y} }$. Here the incident direction $\hat{y}$ is given by a point on the unit sphere/circle denoted $\mathbb{S}^{d-1}$.   In this case one has that the corresponding far-field operator used in the inversion algorithm is injective with a dense range provided that the corresponding wave number is not a transmission eigenvalue \cite{fmconductbc}. One can consider these wave numbers as corresponding to frequencies where there is a Herglotz wave function that does not produce a scattered wave if taken as the illuminating incident wave, where the Herglotz wave function can be seen as a superposition of incident plane waves. Transmission eigenvalue problems are non-linear and non-self-adjoint eigenvalue problems. This makes their investigation interesting mathematically but also challenging. This has lead to new and interesting methods for the theoretical as well as computational study of these problems. 

In general, these eigenvalues can be determined from the far-field data via the linear sampling method (LSM) and the inside-out duality method (IODM). See for the determination of the classical transmission eigenvalues in \cite{TE-book, cchlsm, mypaper1, armin, isp-n2} and \cite{te-cbc,te-cbc2} for numerical examples for the transmission eigenvalues with a conductive boundary condition. This implies that one does not need to have {\it a prior knowledge} of the coefficients to determine the eigenvalues.  Since in \cite{te-cbc} the transmission eigenvalues have been shown to depend on the material parameters monotonically one can study the {\it inverse spectral problem} of determining/estimating the material parameters from the eigenvalues. The most well known inverse spectral problems is the ``Can you hear the shape of a drum'' problem proposed in the 1966 paper \cite{drumshape}. This problem amounts to the question: Do the Dirichlet eigenvalues of the Laplacian uniquely determine the shape of the domain? It is well known that one can not hear the shape but some geometric properties can be uniquely determined. Similarly the question in the preceding section can be seen as ``Can you hear the material parameters?'' since we wish to estimate the material properties from the far-field data (see for e.g. \cite{mypaper1,isp-n2}). In many medical and engineering applications one wants to detect changes in the material parameters of a known object. These problems can fall under the category of non-destructive testing where one wants to infer the integrity of the interior structure using acoustic or electromagnetic waves.

We are motivated by the previous work in \cite{te-cbc} where this eigenvalue problem was first analyzed as well as \cite{te-cbc2} where the investigation was continued with the study of the eigenvalues as the conductivity tends to zero. We also refer to \cite{electro-cbc,two-eig-cbc} for the study of the electromagnetic transmission eigenvalues with a conductive boundary condition. There are some important theoretical and computational questions concerning the transmission eigenvalues with a conductive boundary condition that are studied in this paper. Here we will consider three theoretical questions for this problem, as is done in \cite{DMS-te} for the classical transmission eigenvalues. First we derive Faber-Krahn type inequalities for the transmission eigenvalues with a conductive boundary (see for e.g. \cite{te-fk,te-void}). Then, we prove the discreteness as well as establish the limiting behavior as the conductivity tends to infinity for a sign changing contrast. The limiting behavior as the conductivity tends to infinity for the electromagnetic transmission eigenvalues with a conductive boundary condition was studied in \cite{two-eig-cbc}. Using the limiting behavior we will consider the inverse spectral problem of estimating the refractive index provided that the conductivity is larger but unknown.  Recently, the method of fundamental solutions for computing the classical transmission eigenvalues was studied and implemented in \cite{mfs-te} and in \cite{kleefeldITP} a system of boundary integral equations was established. We will also provide a new boundary integral equation for computing the transmission eigenvalues using the idea given in \cite{cakonikress}. Precisely, the new boundary integral equation is derived from making the ansatz that the eigenfunction can be written in terms of a single layer potential. 

The rest of the paper is ordered as follows. In Section \ref{te-section}, we begin by defining the transmission eigenvalue problem as well as study the analytical questions with a conductive boundary condition. Then in Section \ref{bie}, we will derive a new boundary integral equation for computing the eigenvalues and give details on its implementation. In Section \ref{numerics-section}, we provide verification of the implementation. Additionally, we will numerically validate the theoretical results for the limiting behavior as the conductivity tends to infinity. Lastly, we will numerically investigate the inverse spectral problem of estimating the refractive index from the eigenvalues provided the conductivity is sufficiently large but unknown.

\section{Analysis of the Transmission Eigenvalues }\label{te-section}
In this section, we analytically study the transmission eigenvalues with a conductive boundary condition. To do so, we will begin by rigorously defining the eigenvalue problem in the appropriate function spaces. In our analysis we will use a variational method to prove the theoretical results. Our analysis will also use the concept of $T$ coercivity for a sesquilinear form. This has been used in \cite{te-coersive1,te-coersive2} for studying other transmission eigenvalue problems. In particular, our analysis we will use the $T$ coercivity developed in \cite{T-coercive}. Therefore, we let $D \subset \R^d$ be a bounded simply connected open set with $\nu$ being the unit outward normal to the boundary. Here we will assume that the boundary $\partial D$ is either of class $\mathscr{C}^2$ or is polygonal with no reentrant corners. This assumption on the boundary will allow us to appeal to elliptic regularity estimates in our analysis. Furthermore, we assume that the refractive index $n(x)$ is a scalar bounded real-valued function defined in $D$ and the conductivity parameter $\eta (x)$ is a scalar bounded real-valued function defined on the boundary $\partial D$.

The transmission eigenvalue problem with conductive boundary condition is given by: find $k \in \C$ and non-trivial $(w,v) \in L^2(D) \times L^2(D)$ such that  
\begin{align}
\Delta w +k^2 n w=0 \quad \text{and} \quad \Delta v + k^2 v=0  \quad &\textrm{ in } \,  D \label{teprob1} \\
 w-v=0  \quad  \text{and} \quad  {\partial_{\nu} w}-{\partial_\nu v}= \eta v \quad &\textrm{ on } \partial D \label{teprob2} 
\end{align} 
where $w-v \in  X(D)$. The variational space for difference of the eigenfunctions is 
$$X(D)=H^2(D) \cap H^1_0(D) \quad \text{ such that } \quad \| \cdot \|_{X(D)} = \|\Delta \cdot \|_{L^2(D)}\,.$$
By our assumptions on the boundary $\partial D$ we have that the well-posedness estimate for the Poisson problem and the $H^2$ elliptic regularity estimate (see for e.g. \cite{evans}) implies that the $L^2$ norm of the Laplacian is equivalent to the $H^2$ in the associated Hilbert space $X(D)$. 
Here the Sobolev spaces are given by   
$$H^2(D) =\big\{ \varphi \in L^2(D) \, : \, \partial_{x_i} \varphi \quad  \text{and} \quad \partial_{x_i x_j} \varphi \in L^2(D) \, \text{ for } \, i,j =1, \ldots, d \big\} $$
and 
$$H^1_0(D) =\big\{ \varphi \in L^2(D) \, : \, \partial_{x_i} \varphi  \in L^2(D) \,\,  \text{ for } \,\,  i =1, \ldots , d \,\, \text{ with } \,\, \varphi|_{\partial D}=0 \big\}\,.$$

Notice that by subtracting the equations in \eqref{teprob1} and the boundary conditions in \eqref{teprob2} we have that the difference $u=w-v$ satisfies 
$$ \Delta u +k^2 n u=-k^2(n-1)v  \quad \textrm{ in } \, D\,, \quad u=0 \, \, \textrm{ and } \, \, {\partial_\nu}u =  \eta v \, \textrm{ on } \, \partial  D\,.$$
{\color{black} For analytical considerations we will assume that $|n-1|^{-1} \in L^{\infty}( D)$.} Therefore, we have that the transmission eigenvalue problem with conductive boundary condition can be written as: 
find the values $k \in \C$ such that there is a nontrivial solution $u \in X(D)$ satisfying
\begin{align}
(\Delta+k^2)\frac{1}{n-1}(\Delta u +k^2 n u)=0  \quad &\textrm{ in } \,  D\,, \label{teprobu3} \\
-\frac{k^2}{\eta} \frac{\partial u}{\partial \nu}=\frac{1}{n-1}(\Delta u +k^2 n u) \quad &\textrm{ on } \partial D\,.  \label{teprobu4} 
\end{align}
The boundary condition \eqref{teprobu4} is understood in the sense of the trace theorem. In \cite{te-cbc} it has been shown that \eqref{teprob1}--\eqref{teprob2} and \eqref{teprobu3}--\eqref{teprobu4} are equivalent. We also have that there are infinity many real transmission eigenvalues provided $\eta$ is strictly positive on $\partial D$ and the contrast $n-1$ is either uniformly positive or negative in $D$.
Notice that $v$ and $w$ are related to the eigenfunction $u$ by
$$v = -\frac{1}{k^2(n-1)}(\Delta u + k^2 n u) \quad \text{ and } \quad w = -\frac{1}{k^2(n-1)}(\Delta u + k^2 u)\,. $$
In order to study this eigenvalue problem \eqref{teprobu3}--\eqref{teprobu4} we derive an equivalent variational formulation. Taking a test function $\varphi \in X(D)$, multiplying it by the conjugate of \eqref{teprobu3}, and using Green's Second Theorem we have that 
\begin{align}        
0 =  \int\limits_D \frac{1}{ n-1 }(\Delta u +k^2 nu) (\Delta \overline{\varphi} +k^2 \overline{\varphi}) \, \text{d}x +  \int\limits_{\partial D} \frac{k^2}{\eta} \frac{\partial u}{\partial \nu} \frac{\partial \overline{\varphi} }{\partial \nu} \, \text{d}s \label{TE-varform} 
\end{align}
(see \cite{te-cbc} for details) where the boundary integral is obtained by the conductive boundary condition \eqref{teprobu4}. In order for the variational form to be well-defined in $X(D)$ we will assume that 
$$|n-1|^{-1} \in L^{\infty}( D) \quad \text{ and } \quad 0<\eta_{\text{min}} \leq \eta \leq \eta_{\text{max}} \quad \text{ for a.e.}  \,\,\,\,x \in  D$$
with $\eta_{\text{min}}$ and $\eta_{\text{max}}$ are constants.

\subsection{Faber-Krahn type inequalities}
We now show that there is a lower bound on the real transmission eigenvalues provided $\eta$ is strictly positive on $\partial D$ and the contrast $n-1$ is either uniformly positive or negative in $D$. We will derive these lower bounds by studying the variational formulation \eqref{TE-varform}. The lower bounds will depend on if the contrast is positive or negative in $D$. Here we use similar techniques as in \cite{te-void} where Faber-Krahn inequalities for the classical transmission eigenvalues with a cavity have been derived. 

To begin, we first need a Poincar\'e type estimate in $X(D)$ in terms of an auxiliary eigenvalue problem. Therefore, we let $\lambda \in \R_{+}$ and the non-trivial $\psi \in X(D)$ be the simply supported plate buckling eigenpair with 
\begin{align}        
\Delta^2 \psi = -\lambda \Delta \psi \,\, \,\, \text{in} \,\, \,\,D \quad \text{ and } \quad \Delta \psi = 0  \,\,\,\, \text{{\color{black} on}} \,\,\,\, \partial D\,. \label{aux-eig}
\end{align}
It is clear that there are infinitely many eigenvalues $\lambda_j$ and eigenfunctions $\psi_j$ satisfying the above problem. The eigenvalues also satisfy a Courant-Fischer min-max principle and in particular the smallest eigenvalue satisfies
$$\lambda_1 = \min\limits_{\varphi \in X(D) \setminus \{ 0\} }  \frac{ \| \Delta \varphi \|^2_{L^2(D)} }{ \| \grad \varphi \|^2_{L^2(D)}}\,.$$
Therefore, we can conclude that for all $\varphi \in X(D)$ we have the Poincar\'e type estimate
$$ \lambda_1 \| \grad \varphi \|^2_{L^2(D)} \leq  \| \Delta \varphi \|^2_{L^2(D)}\,.$$
From this we can now derive the Faber-Krahn type inequalities for the transmission eigenvalues with a conductive boundary condition. For this we will now assume that there are two constants  such that 
$$ n_{\text{min}} \leq n \leq n_{\text{max}} \quad \text{ for a.e.}  \,\,\,\,x \in  D\,.$$
\begin{theorem}\label{fk-inequ}
 Let $k$ be a real transmission eigenvalue satisfying \eqref{teprobu3}--\eqref{teprobu4} then we have the inequalities  
$$ k^2 \geq \frac{\lambda_1}{n_{max}} \,\, \textrm{ when } \,\, {\color{black}n_{\text{min}}>1} \quad \text{ and} \quad  k^2 \geq \frac{\lambda_1 {\color{black} \eta_{\text{min}}} }{{\color{black} \eta_{\text{min}}} +C_T \lambda_1} \,\, \textrm{ when } \,\, {\color{black}n_{\text{max}}<1}\,. $$ 
Here $\lambda_1$ is the first simply supported plate buckling eigenvalue for $D$ and $C_T$ is the trace theorem constant such that $\| \partial_{\nu} \varphi  \|^2_{L^2(\partial D)} \leq C_T \| \Delta \varphi \|^2_{L^2(D)}$ for all $\varphi \in X(D)$.
\end{theorem}
\begin{proof} 
To prove the claim we first start with the case when ${\color{black}n_{\text{min}}>1}$. Therefore, it can be shown (see also \cite{te-cbc}) by taking $\varphi =u$ (the corresponding eigenfunction) that \eqref{TE-varform} can be written as 
\begin{align*}        
0&=\int\limits_D  \frac{1}{ n-1} | \Delta u +k^2 u |^2 +k^4 |u|^2 -k^2|\grad u |^2 \, \dif x + k^2 \int\limits_{\partial D} \frac{1}{\eta} |\partial_{\nu} u|^2 \, \dif s \\
  &\geq \int\limits_D  \alpha | \Delta u +k^2 u |^2 \, \dif x+  \int\limits_Dk^4 |u|^2 -k^2|\grad u |^2 \, \dif x + k^2 \int\limits_{\partial D} \frac{1}{\eta} |\partial_{\nu} u|^2 \, \dif s
\end{align*}        
where we let $\alpha=(n_{max}-1)^{-1}>0$. We now expand to obtain
\begin{align*}        
0 &\geq \alpha \| \Delta u \|^2_{L^2(D)} -2\alpha k^2 \| \Delta u \|_{L^2(D)}\| u \|_{L^2(D)} +k^4(\alpha+1) \|  u \|^2_{L^2(D)} -k^2 \| \grad u \|^2_{L^2(D)}\\
   &\geq \left( \alpha -\frac{\alpha^2}{\ep} \right)\| \Delta u \|^2_{L^2(D)} +k^4(\alpha+1-\ep)\|  u \|^2_{L^2(D)}-k^2 \| \grad u \|^2_{L^2(D)}
\end{align*}
by appealing to Cauchy-Schwarz inequality and Young's inequality for any positive $\ep$. Provided that $\ep \in (\alpha , \alpha+1)$ we can conclude that 
\begin{align*}        
0 &\geq \left( \alpha -\frac{\alpha^2}{\ep} \right)\| \Delta u \|^2_{L^2(D)} -k^2 \| \grad u \|^2_{L^2(D)}\\
   &\geq \left( \alpha -\frac{\alpha^2}{\ep}  - \frac{k^2}{\lambda_1}\right)\| \Delta u \|^2_{L^2(D)}
\end{align*}
where we have used the Poincar\'e estimate in $X(D)$. This implies that 
$$ \alpha \lambda_1 \left(1- \frac{\alpha}{\ep} \right) \leq k^2 \quad \text{ for all } \quad \ep \in (\alpha , \alpha+1)$$
provided that  $k$ is a transmission eigenvalue. Letting $\ep \to \alpha+1$ and the fact that $\alpha=(n_{\text{max}}-1)^{-1}>0$ we obtain $k^2 \geq \frac{\lambda_1}{n_{\text{max}}}$ which proves the result for this case. 

Now, for the case when ${\color{black} n_{\text{max}}<1}$ we take $\varphi =u$  and rewriting \eqref{TE-varform} we can obtain 
\begin{align*}        
0&= \int\limits_D  \frac{n}{ 1-n}| \Delta u +k^2 u |^2 +|\Delta u|^2  \, \dif x -k^2 \int\limits_D |\grad u|^2  \, \dif x -k^2 \int\limits_{\partial D} \frac{1}{\eta} |\partial_{\nu} u|^2  \, \dif s\\
  &\geq \| \Delta u \|^2_{L^2(D)} -k^2 \| \grad u \|^2_{L^2(D)} -  \frac{k^2}{ {\color{black} \eta_{\text{min}}} } \| \partial_{\nu} u  \|^2_{L^2(D)}\,. 
\end{align*}        
We now apply the trace theorem as well as the Poincar\'e type estimate in $X(D)$ 
$$ 0\geq  \left(1 - \frac{k^2}{\lambda_1} - \frac{C_T k^2}{{\color{black} \eta_{\text{min}}}} \right) \| \Delta u \|^2_{L^2(D)}\,.$$
Therefore, we again can conclude $k$ to be a transmission eigenvalue. Then a simple calculation gives that $k^2 \geq  \frac{\lambda_1 {\color{black} \eta_{\text{min}}} }{{\color{black} \eta_{\text{min}}} +C_T \lambda_1}$ which proves the claim. 
\end{proof}

Notice that Theorem \ref{fk-inequ} holds for any domain $D$ with piecewise smooth boundary $\partial D$. Here we have assumed that $\partial D$ is either of class $\mathscr{C}^2$ or is polygonal with no reentrant corners. This gives that we have the elliptic $H^2$ regularity estimate. Using this one can show that the simply supported plate buckling eigenvalues coincide with the Dirichlet eigenvalues for the Laplacian. Theorem \ref{fk-inequ} then gives that the transmission eigenvalues with a conductive boundary condition satisfy the same Faber-Krahn inequality given in \cite{te-fk} when ${\color{black} n_{\text{min}} }>1$ as the classical transmission eigenvalues. For the case when ${\color{black} n_{\text{max}}<1}$ we note that from \cite{te-cbc2} that as ${\color{black}\eta \to 0}$ the $k_\eta \to k_0$ where $k_0$ is the classical transmission eigenvalue with $\eta = 0$. Therefore, in the limit as ${\color{black} \eta \to 0}$ the Faber-Krahn inequality in Theorem \ref{fk-inequ} becomes the known inequality in \cite{te-fk} for the classical transmission eigenvalues. 

\subsection{Discreteness for sign changing contrast}
For this section, we will study the discreteness of the transmission eigenvalues when the contrast $|n-1|^{-1} \in L^{\infty}( D)$. Recall, that this is all that is needed to make sense for the variational formulation \eqref{TE-varform}. In \cite{te-cbc} the discreteness of the set of transmission eigenvalues was established provided that the contrast is either uniformly positive or negative in $D$. This is a strong requirement assumed on the contrast, especially if one is interested in the problem of estimating the refractive index from the eigenvalues for non-destructive testing. We will employ the concept of $T$ coercivity that was studied in \cite{T-coercive} for the biharmonic operator with sign changing coefficients.

To begin, we will define what it means for a sesquilinear form to be $T$ coercive. Let $a(\cdot , \cdot )$ be a given bounded sesquilinear form acting on a Hilbert space $X$, then we say that $a(\cdot , \cdot )$ is $T$ coercive provided that there is an isomorphism $T : X \mapsto X$ such that $a(\cdot , T \cdot )$ is a coercive sesquilinear form. Notice that by the inf-sup condition we have that if $a(\cdot , \cdot )$ is $T$ coercive, then $a(\cdot , \cdot )$ can be represented by a continuously invertible operator $A: X \mapsto X$ such that 
$$ a(u,\varphi)=(Au,\varphi)_{X} \quad \text{ for all } \quad u,\varphi \in X\,.$$
This will be used to split the variational formulation \eqref{TE-varform} into an invertible and compact part. Then we can appeal to the Analytic Fredholm Theorem (see for e.g. \cite{TE-book}) to conclude discreteness.

Now, we have that the variational formulation \eqref{TE-varform} can be written as 
\begin{align}
a(u,\varphi)+b_k (u,\varphi) = 0 \quad \text{ for all } \quad \varphi \in X(D)\,. \label{varform}
\end{align}
The bounded sesquilinear forms $a(\cdot \, , \cdot)$ and $b_k(\cdot \, , \cdot)$ on $X(D) \times X(D)$ are given by 
\begin{align}
a(u,\varphi) &=  \int\limits_D  \frac{1}{ n-1}  \Delta u \Delta \overline{\varphi}  \, \dif x   \label{a-form}
\end{align}
and 
\begin{align}
b_k(u,\varphi) &=  k^2\int\limits_D \frac{1}{n-1} ( \overline{\varphi}  \, \Delta u + u \, \Delta \overline{\varphi}) \, \dif x - k^2 \int\limits_D \grad u \cdot \grad \overline{\varphi} \, \dif x \nonumber \\
 &\hspace{2in} + k^2 \int\limits_{\partial D} \frac{1}{\eta} {\partial_{\nu} u}{\partial_{\nu} \overline{\varphi} } \, \dif s + k^4  \int\limits_D  u \overline{\varphi} \, \dif x\,.   \label{b-form}
\end{align}
The boundedness is clear from the fact that it is assumed that $|n-1|^{-1} \in L^{\infty}( D)$ and the conductivity satisfies $0<\eta_{\text{min}} \leq \eta$. Following the arguments in \cite[Section 3]{te-cbc}, we can conclude that $b_k(\cdot \, , \cdot)$ in \eqref{b-form} can be represented by a compact operator $B_k :X(D) \mapsto X(D)$. We also obtain since 
$$ b_k (u,\varphi)=(B_k u,\varphi)_{X(D)} \quad \text{ for all } \quad u,\varphi \in X(D)$$
that $B_k$ depends analytically on $k \in \C$. Similarly, we can have that there is a bounded operator $A:X(D) \mapsto X(D)$ that represents the sesquilinear form $a(\cdot \, , \cdot)$ in \eqref{a-form}. Now, motivated by \cite{T-coercive} we define the operator $T:X(D) \mapsto X(D)$ such that 
\begin{align}
\Delta T\varphi = (n-1) \Delta \varphi \quad \text{ for all} \quad  \varphi \in X(D)\,. \label{T-def}
\end{align}
We now wish to show that the operator $T$ defined by \eqref{T-def} is an isomorphism on $X(D)$ whenever $\partial D$ is either of class $\mathscr{C}^2$ or is polygonal with no reentrant corners. To do so, we first notice that by the well-posedness of the Poisson problem we have that for every $\varphi \in X(D)$ there is a $T\varphi \in H_0^1(D)$ satisfying \eqref{T-def}. Now, by elliptic regularity \cite{evans} we conclude that $T\varphi \in X(D)$. Moreover, by \eqref{T-def} along with the fact that $|n-1|^{-1} \in L^{\infty}( D)$ we easily obtain that 
$$ C \|\Delta \varphi\|_{L^2(D)}^2 \leq  \big| (  \Delta T\varphi ,  \Delta \varphi )_{L^2(D)} \big|$$
where the constant $C$ is independent of $\varphi$ but depend only on  the contrast. Since $\| \cdot \|_{X(D)} = \|\Delta \cdot \|_{L^2(D)}$ we can conclude that $T$ is coercive on $X(D)$. This implies that  $T$ is  an isomorphism on $X(D)$ by the Lax-Milgram Lemma. With this we are ready to prove the discreteness of the transmission eigenvalues.

\begin{theorem}\label{discrete}
Assume that $|n-1|^{-1} \in L^{\infty}( D)$ and $0<\eta_{\text{min}} \leq \eta$. Then the set of transmission eigenvalues is at most discrete. 
\end{theorem}
\begin{proof} 
To prove the claim we will appeal to the Analytic Fredholm Theorem. Therefore, by the definition of the operator $T$ given in \eqref{T-def} we notice that 
$$a(u,Tu) = \int\limits_D  \frac{1}{ n-1}  \Delta u \Delta \overline{T u}  \, \dif x =  \|\Delta u\|_{L^2(D)}^2 \quad \text{ for all } \quad u \in X(D)$$ 
which implies that $a(\cdot , \cdot )$ is $T$ coercive on $X(D)$ and therefore $a(\cdot \, , \cdot)$ can be represented by an invertible operator. From \eqref{varform} we have that $k \in \C$ is a transmission eigenvalue if and only if $A+B_k$ is not injective.  Since $A$ is invertible and $B_k$ is compact we have that $A+B_k$ is Fredholm
with index zero. By definition of $B_k$ from the sesquilinear form given in \eqref{b-form} we have that $B_0$ is the zero operator. We then conclude that $k=0$ is not a transmission eigenvalue since $A+B_0$ is injective and by the Analytic Fredholm Theorem there is at most a discrete set of values in $\C$ where $A+B_k$  fails to be injective, proving the claim. 
\end{proof}

Notice that Theorem \ref{discrete} only requires that $|n-1|^{-1} \in L^{\infty}( D)$. This implies that the contrast can take both positive and negative values. The discreteness result given in \cite{te-cbc} as well as the analysis (for real-valued contrast) in \cite{te-cbc2} depends on the contrast being of one sign. The existence of these transmission eigenvalues is still an open question for the case of a sign changing coefficient. We also note that Theorem \ref{discrete} holds for the case when the refractive index is complex-valued as long as $|n-1|^{-1} \in L^{\infty}( D)$. 

\subsection{Convergence as the conductivity goes to infinity}
In this section, we study the convergence as $\eta$ tends to infinity. Here, we will assume that $\eta \in L^{\infty}(\partial D)$ and satisfies that $\eta_{\text{min}} \to \infty$. This has been studied for the analogous Maxwell's transmission eigenvalues with a conductive boundary condition in \cite{two-eig-cbc}. The analysis in  \cite{two-eig-cbc} for the Maxwell's system is only established when the contrast is either positive or negative definite. By appealing to $T$ coercivity discussed in the previous section we will be able to study the convergence for a sign changing contrast. We also remark that the analysis as $\eta$ tends to zero in \cite{te-cbc2} can be augmented to work for a sign changing contrast as well. 

Now, we assume that the transmission eigenvalues $k_\eta \in \R_{+}$ form a bounded set as $\eta_{\text{min}} \to \infty$. From Theorem 3.2 in \cite{te-cbc2} we have that this is the case provided that contrast is positive (or negative) in the domain $D$. Also, since the eigenfunction $u_\eta$ is non-trivial we may assume that they are normalized in $X(D)$ i.e. $\|\Delta u_\eta \|_{L^2(D)}^2 =1$ for any $\eta$. Since the sequences $(k_\eta , u_\eta) \in \R_{+} \times X(D)$ are bounded as $\eta_{\text{min}} \to \infty$ we then conclude  that there exists $k_\infty $ and $u_\infty $ such that 
$$k_\eta  \to k_\infty  \quad \text{ and } \quad u_\eta \weakc u_\infty  \,\, \textrm{ in } \,\, X(D) \quad \textrm{ as } \,\, \eta_{\text{min}} \to \infty\,.$$
Recall, that the eigenpair satisfies 
$$ a(u_\eta ,\varphi)+b_{k_\eta ,\eta } (u_\eta ,\varphi) = 0  \quad \text{ for all } \quad \varphi \in X(D)$$
where the sesquilinear forms $a(\cdot \, , \cdot)$ and $b_{k ,\eta }(\cdot \, , \cdot)$ are defined as in \eqref{a-form}--\eqref{b-form} where the dependance on $\eta$ is made explicit. Here, we will let $b_{k ,0 }(\cdot \, , \cdot)$ denote the sesquilinear form without the boundary integral on $\partial D$. Therefore, we have that since $k_\eta$ and $u_\eta$ are bounded 
$$ | a(u_\eta ,\varphi)+b_{k_\eta , 0 } (u_\eta ,\varphi) |= \left| k_\eta ^2 \int\limits_{\partial D} \frac{1}{\eta} {\partial_{\nu} u_\eta}{\partial_{\nu} \overline{\varphi} } \, \dif s  \right| \leq \frac{C}{\eta_{\text{min}}} {\color{black} \| \partial_{\nu} \varphi  \|_{L^2(\partial D)} }$$
which tends to zero as $\eta_{\text{min}} \to \infty$ for any $\varphi \in X(D).$ By appealing to the convergence of the eigenvalues and weak convergence of the eigenfunctions we have that
$$ a(u_\infty  ,\varphi)+b_{k_\infty  ,0 } (u_\infty  ,\varphi) =0   \quad \text{ for all } \quad \varphi \in X(D)\,.$$
In order to prove the main result of this section we first must show the convergence of $u_\eta$ to $u_\infty$ in the $X(D)$ norm. 
\begin{lemma}\label{conv-lemma}
Assume that $|n-1|^{-1} \in L^{\infty}(D)$ and $k_\eta \in \R_{+}$ forms a bounded set as $\eta_{\text{min}} \to \infty$. Then $u_\eta \to u_\infty $ in $X(D)$ as $\eta_{\text{min}} \to \infty$. 
\end{lemma}
\begin{proof}
In order to prove the claim we first notice that for any $\varphi \in X(D)$ 
\begin{eqnarray*}
 &&a\big( u_\eta - u_\infty , \varphi \big) \\
&& \hspace{0.4in}= b_{k_\infty , 0 }(u_\infty , \varphi ) - b_{k_\eta , 0 }(u_\infty , \varphi ) - b_{k_\eta , 0 }\big( u_\eta - u_\infty , \varphi \big) - k_\eta ^2 \int\limits_{\partial D} \frac{1}{\eta} {\partial_{\nu} u_\eta}{\partial_{\nu} \overline{\varphi} } \, \dif s. 
\end{eqnarray*}
We now need to estimate the righthand side of the equality and prove that it tends to zero as $\eta_{\text{min}} \to \infty$ for $\varphi = T(u_\eta - u_\infty)$. Then, by the $T$ coercivity we have that
$$ a\big( u_\eta - u_\infty  , T(u_\eta - u_\infty )  \big) = \|\Delta (u_\eta -u_\infty ) \|_{L^2(D)}^2$$
where $T$ is defined via \eqref{T-def} and showing that $ a\big( u_\eta - u_\infty  , T(u_\eta - u_\infty )  \big)$ tends to zero as $\eta_{\text{min}} \to \infty$ will give the result. We begin with 
\begin{eqnarray*}
&&b_{k_\infty  , 0 }(u_\infty  , \varphi ) - b_{k_\eta , 0 }(u_\infty  , \varphi )\\
&& \hspace{0.4in} = (k_\infty ^2 -k_\eta^2)\int\limits_D \frac{1}{n-1} ( \overline{\varphi}  \, \Delta u_\infty  + u_\infty  \, \Delta \overline{\varphi}) \, \dif x  \\
 && \hspace{0.4in}-  (k_\infty ^2 -k_\eta^2) \int\limits_D \grad u_\infty  \cdot \grad \overline{\varphi} \, \dif x   + (k_\infty ^4 -k_\eta^4)  \int\limits_D  u_\infty  \overline{\varphi} \, \dif x 
\end{eqnarray*}
for any $\varphi \in X(D)$. Notice that by since $k_\eta \to k_\infty $ as $\eta_{\text{min}} \to \infty$
$$ \big| b_{k_\infty , 0 }\big( u_\infty , T(u_\eta - u)  \big)  - b_{k_\eta , 0 }\big( u_\infty  , T(u_\eta - u_\infty )  \big)\big| \longrightarrow 0 \quad \textrm{ as } \,\, \eta_{\text{min}} \to \infty$$
where we have used the fact that $T$ is a bounded operator and $u_\eta - u_\infty $ is bounded in $X(D)$.
Just as before we have that 
$$ \left| k_\eta ^2 \int\limits_{\partial D} \frac{1}{\eta} {\partial_{\nu} u_\eta}{\partial_{\nu} \overline{T(u_\eta - u_\infty)} } \, \dif s  \right| \leq \frac{C}{\eta_{\text{min}}} \longrightarrow 0 \quad \textrm{ as } \,\, \eta_{\text{min}} \to \infty$$ 
where we have used the fact that $T$ is a bounded operator and $u_\eta - u_\infty$ is bounded in $X(D)$. Now, notice that 
\begin{eqnarray*}
&&b_{k_\eta , 0 }\big( u_\eta - u_\infty , T(u_\eta - u_\infty)  \big) \\
&& \hspace{0.4in} =  k_\eta^2 \int\limits_D \frac{1}{n-1} \Big( \overline{T(u_\eta - u_\infty) }  \, \Delta (u_\eta - u_\infty) + (u_\eta - u_\infty) \, \Delta \overline{T(u_\eta - u_\infty) } \Big) \, \dif x   \\
 && \hspace{0.4in} -k_\eta^2 \int\limits_D \grad (u_\eta - u_\infty) \cdot \grad \overline{T(u_\eta - u_\infty) } \, \dif x + k_\eta^4  \int\limits_D  (u_\eta - u_\infty) \overline{T(u_\eta - u_\infty) } \, \dif x.  
\end{eqnarray*}
By compact embedding of $H^2(D)$ into $H^1(D)$ we have that 
$$\| u_\eta - u_\infty \|_{H^1(D)} \quad \text{ and } \quad \| T(u_\eta - u_\infty) \|_{H^1(D)}$$ 
tends to zero as $\eta_{\text{min}} \to \infty$. Then by the Cauchy-Schwarz inequality as well as the fact the $k_\eta \in \R_{+}$ forms a bounded set we conclude that 
$$ b_{k_\eta , 0 }\big( u_\eta - u_\infty , T(u_\eta - u_\infty)  \big) \longrightarrow 0 \quad \textrm{ as } \,\, \eta_{\text{min}} \to \infty\,.$$ 
From this we have that 
$$ \|\Delta (u_\eta -u_\infty) \|_{L^2(D)}^2 = a\big( u_\eta - u_\infty , T(u_\eta - u_\infty)  \big) \longrightarrow 0 \quad \textrm{ as } \,\, \eta_{\text{min}} \to \infty$$
which proves the claim. 
\end{proof}

Lemma \ref{conv-lemma} gives that the corresponding eigenpair will converge as $\eta_{\text{min}} \to \infty.$ The natural question is to determine what specific limiting values one obtains for the eigenpair. One would expect the limiting values would be an eigenpair for some associated limiting eigenvalue problem. Therefore, we recall that for a given eigenpair $(k_\eta , u_\eta) \in \R_{+} \times X(D)$ there is a corresponding $(v_\eta ,w_\eta ) \in L^2(D) \times L^2(D)$ such that  
$$v_\eta = -\frac{1}{k_\eta^2(n-1)}(\Delta u_\eta + k_\eta^2 n u_\eta) \quad \text{ and } \quad w_\eta = -\frac{1}{k_\eta^2(n-1)}(\Delta u_\eta + k_\eta^2 u_\eta)\,.$$
Here, $(v_\eta ,w_\eta )$ satisfy the Helmholtz equation and `modified' Helmholtz equation in the distributional sense with 
$$\Delta v_\eta +k_\eta^2 v_\eta=0 \quad \text{and} \quad \Delta w_\eta + k_\eta^2 n w_\eta=0  \quad \textrm{ in } \,\,\,  D\,.$$
Due to the convergence 
$$k_\eta  \to k_\infty  \quad \text{ and } \quad u_\eta \to u_\infty  \,\, \textrm{ in } \,\, X(D) \quad \textrm{ as } \,\, \eta_{\text{min}} \to \infty$$
we can conclude that there is a $(v_\infty ,w_\infty ) \in L^2(D) \times L^2(D)$ such that 
$$v_\eta  \to v_\infty  \quad \text{ and } \quad w_\eta \to w_\infty  \,\, \textrm{ in } \,\, L^2(D) \quad \textrm{ as } \,\, \eta_{\text{min}} \to \infty$$
that are solutions to the Helmholtz equation and `modified' Helmholtz equation in the distributional sense with
$$\Delta v_\infty +k_\infty^2 v_\infty=0 \quad \text{and} \quad \Delta w_\infty + k_\infty^2 n w_\infty=0  \quad \textrm{ in } \,\,\,  D\,.$$
From this we see that it seems that the limiting value $k_\eta^2$ may be either a Dirichlet eigenvalue or `modified' Dirichlet eigenvalue for the domain $D$. The proceeding result proves this assertion. 

\begin{theorem}\label{conv-thm}
Assume that $|n-1|^{-1} \in L^{\infty}(D)$ and $k_\eta \in \R_{+}$ forms a bounded set as $\eta_{\text{min}} \to \infty$. Then $k_\eta \to k_\infty $ as $\eta_{\text{min}} \to \infty$ where $k_\infty^2$ is either Dirichlet eigenvalue or `modified' Dirichlet eigenvalue for the domain $D$.
\end{theorem}
\begin{proof} 
To prove the claim we must show that either $v_\infty$ or $w_\infty$ are non-trivial with zero trace on $\partial D$. Now, recall that 
$$v_\eta = \frac{1}{\eta} \partial_{\nu} u_\eta \quad \text{ and } \quad w_\eta = v_\eta \quad \text{ on } \,\, \partial D\,.$$ 
Since we have assumed that $\|\Delta u_\eta \|_{L^2(D)}^2 =1$ the trace theorem implies that 
$$ \| v_\eta\|_{H^{1/2}(\partial D)} \quad \text{ and } \quad  \| w_\eta\|_{H^{1/2}(\partial D)} $$
tend to zero as $\eta_{\text{min}} \to \infty$. This implies that $v_\infty$ and $w_\infty$ have zero trace on $\partial D$. By appealing to Green's first theorem and the strong $L^2(D)$ convergence, we have that $(v_\infty ,w_\infty ) \in H^1_0(D) \times H^1_0(D)$ such that 
$$\Delta v_\infty +k_\infty^2 v_\infty=0 \quad \text{and} \quad \Delta w_\infty + k_\infty^2 n w_\infty=0  \quad \textrm{ in } \,\,\,  D\,.$$
It is left to prove that either $v_\infty$ or $w_\infty$ are non-trivial. Assume on the contrary that $v_\infty=w_\infty=0$ in $D$. Then we would have that $u_\infty = w_\infty - v_\infty=0$  in $D$ which contradicts the fact that $\|\Delta u_\infty \|_{L^2(D)}^2 =1$ given by Lemma \ref{conv-lemma}. This proves the claim.  

\end{proof}

\section{Boundary Integral Equations}\label{bie}
In this section, we will derive a new boundary integral equations for various interior transmission eigenvalue problems which is based on the idea of \cite{cakonikress}.  Recall, that in general the problem under consideration is the following: Find non-trival solution $(w,v)\in L^2(D) \times L^2(D)$ and $k\in \mathbb{C}$ such that
\begin{align*}
\Delta w +n k^2 w=0 \quad &\textrm{ in } \,  D \\
\Delta v + \tilde{n} k^2 v=0  \quad &\textrm{ in } \,  D  \\
 w=v  \quad &\textrm{ on } \partial D\\
 {\partial_{\nu} w}-{\partial_\nu v}=\eta w \quad &\textrm{ on } \partial D
\end{align*} 
is satisfied, where $n$, $\tilde{n}$ as well as $\eta$ are given. Given different parameters for $\tilde{n}$ and $\eta$ we obtain substantially different eigenvalue problems. Here, we will now assume that the parameters $n$, $\tilde{n}$ {\color{black}and $\eta$} are constant. 
If $\tilde{n}=1$ and $\eta=0$, then we are dealing with the `classical' interior transmission eigenvalue problem.
If $\tilde{n}=1$ and $\eta>0$, then we have the interior transmission eigenvalue problem with conductive boundary condition.
If $\tilde{n}=0$ and $\eta>0$, then we have the zero-index interior transmission eigenvalue problem with conductive boundary condition.
To solve this problem, we use boundary integral equations. We make the single-layer ansatz as in \cite{cakonikress}. Precisely, we use
\[w(x)=\mathrm{SL}_{k\sqrt{n}}\varphi(x)\qquad\text{and}\qquad v(x)=\mathrm{SL}_{k\sqrt{\tilde{n}}}\psi(x)\,,\qquad x\in D\,,\]
where 
\[\mathrm{SL}_{k}\phi(x)=\int_{\partial D}\Phi_k(x,y)\phi(y)\,\mathrm{d}s\,,\qquad  x\in D\]
with $\Phi_k(x,y)$ the fundamental solution of the Helmholtz equation in two dimensions. Here, $\varphi$ and $\psi$ are yet unknown functions on $\partial D$. In order to obtain an integral equation for the  eigenvalue problem we will need the Dirichlet-to-Neumann mapping for the Helmholtz equation as in \cite{cakonikress}. To this end, we will assume that the boundary of the domain is of class $\mathscr{C}^{2,1}$. This gives that on the boundary $\partial D$, we have
\[w(x)=\mathrm{S}_{k\sqrt{n}}\varphi(x)\qquad\text{and}\qquad  v(x)=\mathrm{S}_{k\sqrt{\tilde{n}}}\psi(x)\,,\]
where
\[\mathrm{S}_{k}\phi(x)=\int_{\partial D}\Phi_k(x,y)\phi(y)\,\mathrm{d}s\,,\qquad x\in \partial D\,.\]
Taking the normal derivative and the jump conditions, yields
\[\partial_\nu w(x)=\left(\frac{1}{2}\mathrm{I}+\mathrm{K}'_{k\sqrt{n}}\right)\varphi(x)\qquad\text{and}\qquad  \partial_\nu v(x)=\left(\frac{1}{2}\mathrm{I}+\mathrm{K}'_{k\sqrt{\tilde{n}}}\right)\psi(x)\,,\]
where 
\[\mathrm{K}'_{k}\phi(x)=\int_{\partial D}\partial_{\nu(x)}\Phi_k(x,y)\phi(y)\,\mathrm{d}s\,,\qquad x\in \partial D\,.\]
Hence, using the first boundary condition $w={\color{red}v}$, we obtain the following Dirichlet-to-Neumann mappings (assuming $k\sqrt{n}$ and $k\sqrt{\tilde{n}}$ are not eigenvalues of $-\Delta$ in $D$)
\[\partial_\nu w=\left(\frac{1}{2}\mathrm{I}+\mathrm{K}'_{k\sqrt{n}}\right)\mathrm{S}_{k\sqrt{n}}^{-1}w \qquad\text{and}\qquad  \partial_\nu v=\left(\frac{1}{2}\mathrm{I}+\mathrm{K}'_{k\sqrt{\tilde{n}}}\right)\mathrm{S}_{k\sqrt{\tilde{n}}}^{-1}w\,\]
since the boundary of the domain is of class $\mathscr{C}^{2,1}$(see for e.g. \cite{cakonikress}).
Using the boundary condition  ${\partial_{\nu} w}-{\partial_\nu v}=\eta w$  yields
\[\left[\left(\frac{1}{2}\mathrm{I}+\mathrm{K}'_{k\sqrt{n}}\right)\mathrm{S}_{k\sqrt{n}}^{-1}-\left(\frac{1}{2}\mathrm{I}+\mathrm{K}'_{k\sqrt{\tilde{n}}}\right)\mathrm{S}_{k\sqrt{\tilde{n}}}^{-1}-\eta \mathrm{I}\right]w=0\,,\]
which can be written abstractly as
\begin{eqnarray}
M(k;n,\tilde{n},\eta)w=0\,.
\label{sys}
\end{eqnarray}
That means, that one has to solve a non-linear eigenvalue problem in $k$ for given constants $n$, $\tilde{n}$, and $\eta$. The boundary integral operator arising in (\ref{sys}) are approximated 
by boundary element collocation method using quadratic interpolation (see  \cite{kleefeldITP} for details). The resulting non-linear eigenvalue problem is solved by the Beyn's algorithm \cite{beyn} using a circle centered at $(\mu,0)$ with radius 0.5 as contour in the complex plane.
The contour integral are approximated by 24 equidistant points on this circle. {\color{black}Supportive theoretical results for the approximation of the new boundary integral equation is subject to future research. It can be expected to follow along the same lines given in \cite{ETNA} for the approximation of the boundary integral equation for the solution of the exterior Neumann problem in combination with Beyn's method \cite{beyn} also used in \cite{HabilKleefeld}.}

\section{Computational Experiments }\label{numerics-section}
In this section, we provide extensive numerical experiments to test our implementation, to show the limiting behavior for small and large $\eta$, and to estimate the index of refraction from known interior transmission eigenvalues. In our experiments we implement Beyn algorithm \cite{beyn} to find the eigenvalues in the circle centered at $(\mu,0)$ with radius 0.5 in the complex plane.
\subsection{Testing the boundary element collocation method}\label{part1}
First, we use $n=4$, $\tilde{n}=1$, and $\eta=0$ within (\ref{sys}) to compute the `classical' interior transmission eigenvalues for a domain $D$ being the unit circle. With $\mu=3.1$ and $40$ collocation points, we obtain the 
interior transmission eigenvalues
$2.9026$, $2.9026$, $3.3842$, $3.4121$, and $3.4121$ which are in agreement with the one given in \cite[Figure 6]{mfs-te}. All presented digits are correct. This can also be compared by computing the zeros of the determinant of
\[\left(
\begin{array}{cc}
 J_m(k\sqrt{n}) & -J_m(k)\\
 \sqrt{n}J_m'(k\sqrt{n}) & -J_m'(k)\
\end{array}
\right) \quad \textrm{ for } \,\, m\in \mathbb{N} \cup \{0\}.\]
For the domain $D$ being an ellipse with semi-axis $1$ and $0.8$, we obtain interior transmission eigenvalues $3.1353$, $3.4852$, $3.5473$, and $3.8843$ using $\mu=3.4$ and $40$ collocation points. All digits are correct as one can compare with \cite[Figure 6]{mfs-te}.

Second, we now investigate the computation of the conductive interior transmission eigenvalues via our new boundary integral equation. To {\color{black}begin}, we use $n=4$, $\tilde{n}=1$, and $\eta=1$ within (\ref{sys}) to compute the eigenvalues for the unit sphere. We use $\mu=3.1$ and $160$ collocation points to obtain $2.7741$, $2.7741$, $3.2908$, 
$3.3122$ and $3.3122$ with five digits accuracy. 
The results are in agreement with \cite[Table 8]{te-cbc2}. All reported digits are correct as one can check also by computing the zeros of the determinant of
\[\left(       
\begin{array}{cc}
 J_m(k\sqrt{n}) & -J_m(k)\\
 k\sqrt{n}J_m'(k\sqrt{n})-\eta J_m(k\sqrt{n}) & -k J_m'(k)\
\end{array}
\right) \quad \textrm{ for } \,\, m\in \mathbb{N} \cup \{0\}.\]
We obtain $3.0034$, $3.3565$, $3.7819$, and $3.4485$ for an ellipse with semi-axis $1$ and $0.8$ when using $\mu=3.3$ and $160$ collocation points.

Third, we compute zero-index conductive interior transmission eigenvalues; that is, we choose the parameters $n=4$, $\tilde{n}=0$, and $\eta=1$ in (\ref{sys}). We obtain $1.7840$, $2.4735$, $2.4735$, $3.1151$, $3.1151$, and $3.4363$ when using 
$\mu=2.0$ and $\mu=3.0$ with 
$160$ collocation points. The results are in agreement with the zeros of the determinant
\[\left(       
\begin{array}{cc}
 J_m(k\sqrt{n}) & -1\\
 k\sqrt{n}J_m'(k\sqrt{n})-\eta J_m(k\sqrt{n}) & -m\
\end{array}
\right) \quad \textrm{ for } \,\, m\in \mathbb{N} \cup \{0\}\]
see \cite[p. 22]{two-eig-cbc} for details.

\subsection{Limiting behavior for small and large conductivity}
First, we see what happens if we let $\eta$ approach zero for the conductive interior transmission problem using a unit circle with $n=4$. From \cite{te-cbc2} we have that as $\eta$ approach zero the conductive interior transmission eigenvalues converge to the `classical' interior transmission eigenvalues. As expected the convergence is linear as one can also see in Table \ref{table1} {\color{black}by checking the estimated order of convergence (EOC) which is given by }
$${\color{black}\mathrm{EOC}=\log\left(\epsilon_\eta/\epsilon_{\eta/2}\right)/\log(2) \quad \text{where} \quad \epsilon_\eta = |k_1(\eta)-k_1(0)|.}$$
\begin{table}[!ht]   
\centering
 \begin{tabular}{l|cc}
  $\eta$ & $k_1(\eta)$ & EOC\\
  \hline
1/2    & 2.8416 &        \\
1/4    & 2.8730 &        \\
1/8    & 2.8880 & 1.0621 \\
1/16   & 2.8954 & 1.0331 \\
1/32   & 2.8990 & 1.0170 \\
1/64   & 2.9008 & 1.0086 \\
1/128  & 2.9017 & 1.0043 \\
1/256  & 2.9022 & 1.0022 \\
1/512  & 2.9024 & 1.0011 \\
1/1024 & 2.9025 & 1.0005
 \end{tabular}
  \caption{\label{table1}Linear convergence of the first conductive interior transmission eigenvalue for $\eta\rightarrow 0$ to the first `classical' interior transmission eigenvalue for the unit circle with $n=4$.}
\end{table}
Next, we let $\eta$ approach $\infty$ in the conductive interior transmission problem. We again use a unit circle with $n=4$. 
\begin{table}[!ht]   
\centering
 \begin{tabular}{r|cc|cc}
  $\eta$ & $k_1(\eta)$ & EOC & $k_2(\eta)$ & EOC\\
  \hline
80    & 2.5998 &      &3.7756&      \\
160   & 2.5838 &      &3.8061&      \\
320   & 2.5758 &0.9959&3.8194&1.2048\\
640   & 2.5718 &0.9983&3.8256&1.0776\\
1280  & 2.5698 &0.9992&3.8287&1.0345\\
2560  & 2.5688 &0.9996&3.8302&1.0163\\
5120  & 2.5683 &0.9998&3.8310&1.0079\\
10240 & 2.5681 &0.9999&3.8313&1.0039\\
20480 & 2.5679 &1.0000&3.8315&1.0019\\
40960 & 2.5679 &1.0014&3.8316&1.0024
 \end{tabular}
  \caption{\label{table2}Linear convergence of the first two conductive interior transmission eigenvalue for $\eta\rightarrow \infty$ towards interior Dirichlet eigenvalues for the unit circle with $n=4$.}
\end{table}
As we can see in Table \ref{table2}, we obtain a linear convergence towards $2.567\,811\,150\,920\,341$ which is the `modified' Dirichlet eigenvalue of $\Delta u+\tau u=0$ with $\tau=2\lambda$ 
(first zero of the Bessel function of the first kind of order two divided by two) 
and towards $3.831\,705\,970\,207\,512$ which is the Dirichlet eigenvalue of  $\Delta u+\lambda u=0$ (first zero of the Bessel function of the first kind of order one). Recall, that the convergence presented in Table \ref{table2} is predicted by Theorem \ref{conv-thm} but the convergence rate has yet to be justified.

Likewise, we can check the convergence as $\eta$ tends to infinity when considering the zero-index conductive interior transmission problem. We use the unit circle with $n=4$ and obtain a linear convergence against the `modified' Dirichlet eigenvalues as shown in Table \ref{table3} which is predicted by Theorem 3.6 of \cite{two-eig-cbc} but the convergence rate has yet to be justified.
\begin{table}[!ht]   
\centering
 \begin{tabular}{r|cc|cc|cc}
  $\eta$ & $k_1(\eta)$ & EOC & $k_2(\eta)$ & EOC & $k_3(\eta)$ & EOC\\
  \hline
80    & 1.9396 &      &2.5993&      &3.2287&      \\
160   & 1.9278 &      &2.5837&      &3.2097&      \\
320   & 1.9218 &0.9919&2.5758&0.9778&3.2000&0.9638\\
640   & 1.9188 &0.9963&2.5718&0.9894&3.1950&0.9825\\
1280  & 1.9173 &0.9982&2.5698&0.9948&3.1926&0.9914\\
2560  & 1.9166 &0.9991&2.5688&0.9974&3.1913&0.9957\\
5120  & 1.9162 &0.9996&2.5683&0.9987&3.1907&0.9979\\
10240 & 1.9160 &0.9998&2.5681&0.9994&3.1904&0.9989\\
20480 & 1.9159 &0.9999&2.5679&0.9997&3.1902&0.9995\\
40960 & 1.9159 &1.0014&2.5679&1.0013&3.1902&1.0011
 \end{tabular}
  \caption{\label{table3}Linear convergence of the first three zero-index conductive interior transmission eigenvalue for $\eta\rightarrow \infty$ towards the `modified'  Dirichlet eigenvalues for the unit circle with $n=4$.}
\end{table}
The three `modified' Dirichlet eigenvalues are given as $1.915\,852\,985\,103\,756$, $2.567\,811\,150\,920\,341$, and $3.190\,080\,947\,961\,992$ which are the zeros of the Bessel function of the first kind or order 1, 2, and 3 
divided by two, respectively. Note that we also tried different $n$, for example $n=3$. We obtain for $\eta=40960$ $2.2123$, $2.9651$, $3.6837$, $4.3812$, and $4.9964$ which are in agreement with the  `modified'  Dirichlet eigenvalues $2.212\,236\,473\,354\,803$, $2.965\,052\,918\,423\,964$, $3.683\,588\,188\,085\,106$, $4.381\,131\,547\,263\,831$, and $4.996\,232\,140\,012\,951$. 
If we use an ellipse with semi-axis $1$ and $0.8$, we obtain a linear convergence as well. The values for $\eta=40960$ are $1.3602$ and $1.5706$ for $n=4$ and $n=3$, respectively. They agree with the  `modified' Dirichlet eigenvalues which are given as $1.3601$ and $1.5705$.

Now, we focus on a scattering object that has different piecewise index of refractions. We consider a circle with radius $R$ and inside there is a circle with radius $r$. 
The interior circle has contrast $n_1$ and the remaining part of the circle has contrast $n_2$. In the sequel, we call this a double layer circle.
Following the approach of \cite[p. 198]{cosso} one needs to compute the zeros of the determinant 
\[\left(
\begin{array}{cccc}
 \scriptscriptstyle -J_m(kR)   & \scriptscriptstyle H_m^{(1)}(kR\sqrt{n_2})                                          & \scriptscriptstyle H_m^{(2)}(kR\sqrt{n_2})                                          & \scriptscriptstyle 0\\
 \scriptscriptstyle -kJ_m'(kR) & \scriptscriptstyle k\sqrt{n_2}H_m^{(1)'}(kR\sqrt{n_2})-\eta  H_m^{(1)}(kR\sqrt{n_2})& \scriptscriptstyle k\sqrt{n_2}H_m^{(2)'}(kR\sqrt{n_2}) -\eta H_m^{(2)}(kR\sqrt{n_2})& \scriptscriptstyle 0\\
 \scriptscriptstyle 0          & \scriptscriptstyle H_m^{(1)}(kr\sqrt{n_2})                                          & \scriptscriptstyle H_m^{(2)}(kr\sqrt{n_2})                                          & \scriptscriptstyle -J_m(kr\sqrt{n_1})\\
 \scriptscriptstyle 0          & \scriptscriptstyle k\sqrt{n_2}H_m^{(1)'}(kr\sqrt{n_2})                              & \scriptscriptstyle k\sqrt{n_2}H_m^{(2)'}(kr\sqrt{n_2})                              & \scriptscriptstyle -k\sqrt{n_1}J_m'(kr\sqrt{n_1})
\end{array}
\right)\]
for $m\in \mathbb{N} \cup \{0\}$ to obtain the conductive interior transmission eigenvalues. 
We use $R=1$, $r=1/2$, $n_1=1/2$, $n_2=4$ to model a sign changing contrast. With $m\in \mathbb{N} \cup \{0\}$ we obtain the results in Table \ref{table4}. 
\begin{table}[!ht]
\centering
 \begin{tabular}{r|cc|cc|cc}
  $\eta$ & $k_1(\eta)$ & EOC & $k_2(\eta)$ & EOC & $k_3(\eta)$ & EOC\\
  \hline
80    & 2.3772 &      &3.2053&      &3.4197&      \\
160   & 2.3904 &      &3.1992&      &3.4124&      \\
320   & 2.3975 &0.9076&3.1975&1.8156&3.4094&1.3075\\
640   & 2.4011 &0.9547&3.1969&1.6568&3.4081&1.1942\\
1280  & 2.4030 &0.9776&3.1967&1.4819&3.4075&1.1111\\
2560  & 2.4039 &0.9888&3.1966&1.3165&3.4072&1.0598\\
5120  & 2.4044 &0.9944&3.1966&1.1881&3.4071&1.0311\\
10240 & 2.4046 &0.9972&3.1966&1.1040&3.4070&1.0159\\
20480 & 2.4047 &0.9986&3.1966&1.0549&3.4070&1.0080\\
40960 & 2.4048 &1.0007&3.1966&1.0296&3.4069&1.0054
 \end{tabular}
  \caption{\label{table4} Linear convergence of the first three conductive interior transmission eigenvalue for $\eta\rightarrow \infty$ towards either the Dirichlet eigenvalues or the `modified' Dirichlet eigenvalues for the double layer circle with $n_1=1/2$ and $n_2=4$.}
\end{table}
As can see in Table \ref{table4}, the values converge for $\eta\rightarrow\infty$ towards either the Dirichlet eigenvalue or the `modified' Dirichlet eigenvalue. Similarly, we can calculate the zero-index conductive interior transmission eigenvalues for a double layer circle by computing the zeros of the determinant
\[\left(
\scriptscriptstyle
\begin{array}{cccc}
 \scriptscriptstyle -R^{m}    &\scriptscriptstyle H_m^{(1)}(kR\sqrt{n_2})                                          &\scriptscriptstyle H_m^{(2)}(kR\sqrt{n_2})                                          &\scriptscriptstyle 0\\
 \scriptscriptstyle -mR^{m-1} &\scriptscriptstyle k\sqrt{n_2}H_m^{(1)'}(kR\sqrt{n_2})-\eta  H_m^{(1)}(kR\sqrt{n_2})&\scriptscriptstyle k\sqrt{n_2}H_m^{(2)'}(kR\sqrt{n_2}) -\eta H_m^{(2)}(kR\sqrt{n_2})&\scriptscriptstyle 0\\
 \scriptscriptstyle 0         &\scriptscriptstyle H_m^{(1)}(kr\sqrt{n_2})                                          &\scriptscriptstyle H_m^{(2)}(kr\sqrt{n_2})                                          &\scriptscriptstyle -J_m(kr\sqrt{n_1})\\
 \scriptscriptstyle 0         &\scriptscriptstyle k\sqrt{n_2}H_m^{(1)'}(kr\sqrt{n_2})                              &\scriptscriptstyle k\sqrt{n_2}H_m^{(2)'}(kr\sqrt{n_2})                              &\scriptscriptstyle -k\sqrt{n_1}J_m'(kr\sqrt{n_1})
\end{array}
\right)\]
for $m\in \mathbb{N} \cup \{0\}$. Using the same parameters as before yields the results in Table \ref{table5} showing that the zero-index conductive interior transmission eigenvalues converge to the `modified' Dirichlet eigenvalues as $\eta$ tends to infinity.
\begin{table}[!ht]   
\centering
 \begin{tabular}{r|cc|cc|cc}
  $\eta$ & $k_1(\eta)$ & EOC & $k_2(\eta)$ & EOC & $k_3(\eta)$ & EOC\\
  \hline
80    & 3.1214 &      &3.1301&      &3.3478&      \\
160   & 3.1225 &      &3.1648&      &3.3774&      \\
320   & 3.1228 &1.5965&3.1825&0.9740&3.3930&0.9192\\
640   & 3.1229 &1.4352&3.1909&1.0675&3.4009&0.9925\\
1280  & 3.1230 &1.2811&3.1946&1.2060&3.4046&1.0850\\
2560  & 3.1230 &1.1645&3.1960&1.3974&3.4062&1.2368\\
5120  & 3.1230 &1.0899&3.1964&1.6099&3.4067&1.4643\\
10240 & 3.1230 &1.0472&3.1965&1.8080&3.4069&1.7639\\
20480 & 3.1230 &1.0242&3.1966&1.9958&3.4069&2.2117\\
40960 & 3.1230 &1.0137&3.1966&2.2537&3.4069&3.9809
 \end{tabular}
  \caption{\label{table5}Linear convergence of the first three zero-index conductive interior transmission eigenvalue for $\eta\rightarrow \infty$ towards the `modified' Dirichlet eigenvalues for the double layer circle with 
  $n_1=1/2$ and $n_2=4$.}
\end{table}

\subsection{Estimation of the index of refraction}
As we have seen in Table \ref{table1} the `classical' transmission eigenvalues are obtained when $\eta$ approaches zero in the conductive interior transmission problem. Here,  we let $k(n ; \eta)$ denote the transmission eigenvalue for refractive index $n$ and conductivity $\eta$. From the convergence results as $\eta$ tends to zero, we have the conductive transmission eigenvalues tend to the classical transmission eigenvalues. When $\eta$ tends to infinity, the zero-index conductive transmission eigenvalues tend to the `modified' Dirichlet eigenvalues. From this, we now show that the refractive index can be estimated provided $\eta$ is sufficently small and unknown. 

One can estimate the index of refraction by solving for $n_\mathrm{approx}$ in the non-linear equation $k_1(n;\eta)=k_1(n_\mathrm{approx})$ for given small $\eta$ where $k_1(n)$ denotes the first classical transmission eigenvalue. 
For the unit circle with $n=4$ and $\eta=1/2$ respectively $\eta=1/10$, to obtain $n_\mathrm{approx}=3.897\,441\,361\,498\,941$ 
and $n_\mathrm{approx}=3.979\,992\,664\,293\,090$, respectively.
We have seen in Table \ref{table3} that the zero-index conductive transmission eigenvalues converge linearly towards the `modified' Dirichlet eigenvalues denoted $\tau(n)$ as $\eta$ tends to infinity. That means, we can estimate the  index of refraction by solving the non-linear equation $k_1(n;\eta)=\tau_1(n_\mathrm{approx})$ for large $\eta$. 
In our examples we take for example $\eta=100$, $\eta=200$, and $\eta=1000$.
Using $n=4$ and these $\eta$'s for the unit circle gives the approximated index of refraction $3.921\,606\,411\,761\,363$, $3.960\,400\,848\,027\,610$, and $3.992\,016\,007\,078\,786$, respectively.
Likewise using $n=3$ with the same $\eta$'s for the unit circle yields the approximated index of refraction $2.941\,204\,808\,821\,021$, $2.970\,300\,636\,020\,707$, and $2.994\,012\,005\,309\,089$. 

\section{Summary and Conclusions }\label{end-section}
In this paper, we have further studied the interior conductive transmission eigenvalue problem both analytically and numerically. The numerical experiments give that the new boundary integral equation developed here can be used to compute multiple interior transmission eigenvalue problems with a constant refractive index. Notice that theoretically the conductivity parameter can be non-constant in the proposed boundary integral equation. We have also investigated the inverse spectral problem numerically for recovering constant refractive indices from the given interior conductive transmission eigenvalue provided $\eta$ is small and unknown. Analytically we have established discreteness as well as the limiting behavior as $\eta \to \infty$ for the interior conductive transmission eigenvalue problem for a sign changing contrast. This analysis and numerical approach are both new and provide a deeper understanding of these eigenvalue problems. One unanswered question is whether the convergence rate as $\eta \to \infty$ for the eigenvalues and eigenfunctions is linear. Another one is the existence of complex eigenvalues provided the refractive index and conductivity parameters are real-valued.


\end{document}